\documentclass[runningheads]{llncs}
\usepackage[T1]{fontenc}
\usepackage{graphicx}
\usepackage{amsmath,amssymb}
\usepackage{mathrsfs}
\usepackage{xspace}
\usepackage{xcolor}
\usepackage{tikz-cd}
\usetikzlibrary{decorations.pathreplacing}
\usetikzlibrary{decorations.markings,arrows.meta}
\usetikzlibrary{fit,backgrounds}
\tikzset{
	midarrow/.style={
		postaction={
		decorate,
		decoration={markings,mark=at position 0.9 with {\arrow{>[length=6pt,width=6pt]}}}
		}
	}
}

\newtheorem{thm}{Theorem}
\newtheorem{cor}[thm]{Corollary}
\newtheorem{lem}[thm]{Lemma}
\newtheorem{prop}[thm]{Proposition}

\usepackage{enumitem}
\setlist[itemize]{label=\textbullet}

\usepackage{subcaption}
\newcommand{\grosmot}[1]{\raisebox{0pt}[0pt][0pt]{#1}}

\newcommand{\ie}{\textit{i.e.}\xspace}

\renewcommand{\leq}{\leqslant}

\renewcommand{\epsilon}{\varepsilon}
\renewcommand{\phi}{\varphi}
\renewcommand{\rho}{\varrho}
\newcommand{\B}{\ensuremath{\mathbb{B}}\xspace}

\makeatletter
\NewDocumentCommand{\textsubstack}{O{c}m}{%
  \mbox{%
    \linespread{1}
    \fontsize{\sf@size}{\fpeval{1.1*\sf@size}}\selectfont
    \begin{tabular}{@{}#1@{}}#2\end{tabular}%
  }%
}
\NewCommandCopy{\amsmathxrightarrow}{\xrightarrow}
\RenewDocumentCommand{\xrightarrow}{om}{%
  \IfNoValueTF{#1}{%
    \amsmathxrightarrow{#2}%
  }{%
    \amsmathxrightarrow[{#1}]{#2}%
  }%
}
\makeatother

\usepackage{bigdelim}
\newcommand{\ASTG}{\ensuremath{\mathrm{ASTG}}}
\usepackage{bm}
\newcommand{\attractors}{\ensuremath{\bm{\mathcal{A}}}}
\newcommand{\chemin}[1]{\xrightarrow[\vphantom{*}\smash{\raisebox{0.3ex}{$*$}}]{#1}}
\newcommand{\textwedge}{\ensuremath{\wedge}\xspace}
\newcommand{\textvee}{\ensuremath{\vee}\xspace}
\newcommand{\textneg}{\ensuremath{\neg}\xspace}

\begin{document}
\title{Strong modules and asynchronous attractors of Boolean networks}
\author{Paul Ruet}
\institute{CNRS, Université de Paris, IRIF, Paris, France \\
\email{ruet@irif.fr} 
}
\maketitle

\begin{abstract}
We consider Boolean networks with interaction graphs partitioned into strongly connected components, which we call strong modules. This type of network decomposition has been considered in the literature, primarily from the perspective of attractor detection algorithms. In this paper, we aim to provide an algebraic basis for this line of research in the case of asynchronous Boolean networks. We prove that the asynchronous attractors of a network can be described as a dependent sum construction: as products of attractors of its controlled strong modules. We then show that a representation of all attractors can be computed in polynomial time under two conditions: the strong modules are small, and either the network is sparse or its defining functions have small size circuits (in particular when they are nested canalizing). We illustrate these results on a published Boolean model.
\end{abstract}

\section{Introduction}


Introduced by von Neumann in the context of automata theory \cite{vN66}, Boolean networks are widely used as simplified models of various biological systems \cite{MP43,Kau93}, in particular gene regulatory networks, where each gene is represented as a binary variable indicating whether it is active or not. The interaction of $n$ variables (genes) is represented by a graph $G$ with $n$ vertices, and the dynamical evolution of each variable $v$ is encoded through a Boolean function with inputs the in-neighbourhood of $v$.

Despite their apparent simplicity, the exact dynamics of Boolean networks of high dimension $n$ is generally intractable. In this paper, we are interested in viewing interaction graphs $G$ as composed of smaller subgraphs, or modules, and we aim to relate the dynamics of these building blocks to the global network's dynamics. In particular, we wish to recover the global dynamics when these modules have tractable dynamics and sufficiently nice interactions between them.

This question of modularity of gene networks has been considered in recent literature, with a focus on identifying biologically meaningful modules (subsets of genes and interactions that carry out relatively self-contained functions) and algorithmic issues (see \cite{IM07,TC13,ZKF13,DAE16,JM21} and references therein). A modular structure also emerges in small-world and scale-free graphs, which are frequently used to model gene regulation \cite{WS98,BA99,Alo07}.

From a mathematical point of view, strongly connected components of the interaction graph $G$ are natural candidates for modules. This is the choice made explicitely in \cite{TC13,ZKF13,MPQY19,KWVML23,MVDKWL25}, and it is the choice we make in this paper too, although future work should likely consider less restrictive notions of biological modules. Strongly connected components are also generalizations of cycles, and studying the influence of components on dynamics fits within the line of research initiated by the biologist R. Thomas, relating the dynamical properties of a network to the presence of positive or negative cycles in its interaction graph \cite{Tho81,RR08b,Rue17}.

Attractors are a central dynamical property, capturing the long-term behavior of the system, often interpreted as stable biologically functional states. The precise definition however depends on the type of Boolean dynamics. In synchronous dynamics, all possible updates occur simultaneously. In asynchronous dynamics, which are often considered more biologically realistic, only one possible update occurs at a time. In this paper, we are interested in asynchronous dynamics.

In Section \ref{sec:decomposition}, we consider a network with interaction graph $G$ split into two vertex sets $I$ and its complement $V(G)\setminus I$, such that $G$ has no edge from $V(G)\setminus I$ to $I$. Such a decomposition induces a network $N_2$ on $V(G)\setminus I$, which is (in general) ``controlled'' by a network $N_1$ on $I$. The network $N_2=N_2(x)$ is controlled in the sense that, in state $x$, the possible transitions of variables in $V(G)\setminus I$ depend on the $I$-coordinates of $x$. Our main result (Theorem \ref{thm:attractors}) states that, when $x$ runs through an attractor $a_1$ of $N_1$, the control of $N_2=N_2(a_1)$ is much smoother: then the ($V(G)\setminus I$)-transitions form an attractor $a_2$ of $N_2(a_1)$. Actually, the asynchronous attractors of the whole network can be described as a dependent sum construction: more precisely, as Cartesian products $a_1\times a_2$ of an attractor $a_1$ of $N_1$ and an attractor $a_2$ of $N_2(a_1)$ (we simplify the notation here).

In Section \ref{sec:gendecomposition}, we apply Theorem \ref{thm:attractors} to the partition of the interaction graph $G$ in strongly connected components, which we call \emph{strong modules}, and we give (Theorem \ref{thm:condensation}) a general formula for the attractors of a network in as Cartesian products of attractors of its controlled strong modules. In \cite{KWVML23,MVDKWL25}, a similar formula is given for synchronous attractors, although the mathematical details appear to be rather different.

The relation between strong module decomposition and asynchronous attractors is studied in \cite{MPQY19} with a view toward attractor detection algorithms. Given a decomposition as above, the authors define \emph{realisations}, which are overlapping subnetworks, on $I$ and on a superset $J\cup(V(G)\setminus I)$ of $V(G)\setminus I$, which are combined by gluing along overlapping variables in $J$. In our Theorem \ref{thm:attractors}, subnetworks do not no overlap, and our remark above on the control by attractors leads to a quite simple algebraic interpretation of attractor sets, and to Section \ref{sec:complexity} on the complexity of computing attractors.

In Section \ref{sec:complexity}, we show indeed that a representation of all attractors can be computed in polynomial (instead of exponential) time under two conditions: first, the strong modules are bounded (independently of dimension $n$); second, either the network is sparse (in the sense of bounded in-degree) or its defining functions have polynomial size Boolean circuits. The latter condition holds in particular when the Boolean functions are nested canalizing (Section \ref{sec:canal}). Recall that the concept of canalization was proposed by Waddington \cite{Wad42} as the property of a biological process of being able to produce a relatively stable phenotype despite the presence of variability. Nested canalizing functions form a class of Boolean functions introduced by Kauffman \cite{Kau93,Kau03} that formalize this canalizing behaviour observed in gene regulatory networks, and recent work \cite{Sub22} gives evidence that nested canalizing functions are predominant in Boolean gene regulatory networks.

In Section \ref{sec:bioex}, we conclude by exploring a Boolean model of a signaling network controlling cell-cycle progression presented in \cite{SFLKBMMSCTPWBA09}. We show the hierarchical structure of attractors present in this model, and compute its $3$ attractors which turn out to be fixed points. Although this illustrates how to apply our results to a realistic example, further work clearly includes the study of more gene network models from the literature, such as those listed in \cite{KWVML23} and exhibiting strong modular structure, in light of the theory presented here.

\section{Boolean networks}

Given a directed graph $G$, $V(G)$ and $E(G)$ denote respectively its sets of vertices and edges. The in-neighborhood (resp. out-neighborhood) of a vertex $v\in V(G)$ is denoted by $G^-(v)$ (resp. $G^+(v)$).

Let $\B=\{0,1\}$. We define a \emph{Boolean network} as a tuple $(G,(f_v)_{v\in V(G)})$ where $G$ is a directed graph and for each $v\in V(G)$, $f_v$ is a function from \grosmot{$\B^{G^-(v)}$} to $\B$. We shall simply denote the family $(f_v)_{v\in V(G)}$ by $f$ and the network by $(G,f)$, and call $G$ the \emph{interaction graph} of $(G,f)$. The cardinality of $V(G)$ is the \emph{dimension} of $(G,f)$.

Note that we do not require that all edges in $G$ are so-called \emph{functional interactions}, so that our interaction graphs are a slight abuse of terminology: indeed, it may happen that for some $v$, $f_v$ does not depend on some $u\in G^-(v)$, \ie for all \grosmot{$x\in\B^{G^-(v)}$}, $f_v(x)=f_v(x^u)$, where $x^u$ denotes the state $x$ with coordinate $u$ flipped. This choice simplifies the presentation, and it is morally harmless, since we are interested in strongly connected components of the interaction graph (see Section \ref{sec:asyncattract}), and the subgraph of functional interactions in $G$ has only smaller components than $G$.

\subsection{Asynchronous state transition graph}
\label{sec:astg}

The \emph{asynchronous state transition graph} associated to a Boolean network $N=(G,f)$ is the directed graph $\ASTG(N)$ with vertex set $\B^{V(G)}$ and a directed edge from $x$ to $y$ when $d(x,y)=1$ and for some $v\in V(G)$,
$$
f_v(x\restriction G^-(v))=y_v\neq x_v.
$$
Here, $d$ denotes the Hamming distance, and the notation $x\restriction I\in\B^I$ denotes the restriction of $x$ to coordinates in $I$, \ie the composite
$$
x\restriction I:I \xrightarrow{\subseteq} V(G) \xrightarrow{x} \B.
$$
The asynchronous state transition graph $\ASTG(N)$ is therefore a \emph{partial orientation} of the hypercube $\B^{V(G)}$, and it is folklore that Boolean networks are in bijective correspondence with such partial orientations. We shall use the notation
$$
x\xrightarrow{N}y \quad \text{(resp. $x\chemin{N}y$)}
$$
to denote that $\ASTG(N)$ has an edge (resp. a directed path) from $x$ to $y$.

\subsection{Asynchronous attractors}
\label{sec:asyncattract}

Recall that the \emph{(asynchronous) attractors} of a Boolean network $N$ are the terminal strongly connected components of its asynchronous state transition graph $\ASTG(N)$. We let $\attractors(N)$ denote the set of attractors of $N$.

\section{Decompositions of Boolean networks}
\label{sec:decomposition}

Let $N=(G,f)$ be a Boolean network, and assume there is a non-empty strict subset $I\subset V(G)$ such that $G$ has no edge from $I$ to $V(G)\setminus I$. We call the ordered pair $(I,V(G)\setminus I)$ a \emph{decomposition} of $N$. It induces a network $N[I]$ and a family of networks $N(I,x)$ for $x\in\B^I$, which we describe below.

Let then $G[I]$ be the subgraph of $G$ induced by $I$ and $N[I]=(G[I],(f_v)_{v\in I})$ be the induced network. For each $x\in\B^I$ and each $v\in V(G)\setminus I$, let $f^x_v:\B^{V(G)\setminus I}\to\B$ be defined by
$$
f^x_v(y)=f_v(x,y).
$$
Here $(x,y)$ denotes the unique $z\in\B^{G^-(v)}$ such that
\begin{align*}
x &= z\restriction G^-(v)\cap I \\
y &= z\restriction G^-(v)\cap(V(G)\setminus I).
\end{align*}
This generally implies a reordering of coordinates and is therefore an abuse of notation.

Note that $f^x_v$ may depend on fewer coordinates: some $u\in V(G)\setminus I$ may be such that $f_v(x,y)=f_v(x,y^u)$ for all $y\in\B^{V(G)\setminus I}$, even though $f_v(x',y)\neq f_v(x',y^u)$ for some $x'\in\B^I$. But our choice of enabling non-functional interactions exempts us from having to worry about these lost coordinates.

Let $G(I,x)$ be the induced subgraph of $G[V(G)\setminus I]$, and $N(I,x)$ be the network $(G(I,x),(f^x_v)_{v\in V(G)\setminus I})$.

\begin{prop}
\label{lem:indep}
If $G$ has no edge from $I\subseteq V(G)$ to $V(G)\setminus I$, then $N(I,x)=N[V(G)\setminus I]$ for any $x\in\B^I$.
\end{prop}

If $a\subseteq\B^I$, let $N(I,a)$ be the (unique) network such that $\ASTG(N(I,a))$ is the edge union of the directed graphs $\ASTG(N(I,x))$ for $x\in a$.

\subsection{Decompositions and reachability}
\label{sec:reach}

Lemma \ref{lem:reachability} shows that the paths from $y\in\B^{V(G)\setminus I}$ in $\ASTG(N(I,x))$ correspond exactly to the paths in $\ASTG(N)$ starting from $(x,y)\in\B^I\times\B^{V(G)\setminus I}$ and keeping $x$ fixed: $x$ behaves as a state controlling $V(G)\setminus I$. On the other hand, $I$ is uncontrolled: paths in $\ASTG(N[I])$ starting from $x$ are in bijective correspondence with paths in $\ASTG(N)$ starting from $(x,y)$ and keeping $y$ fixed.

As we shall see in Section \ref{sec:att}, the situation of an attractor instead of a single controlling state is quite different, and actually simpler. Letting $x$ move inside an attractor $a$ of $N[I]$ gives much more freedom to $y$: as any two points of $a$ are reachable from each other, $y$ can move along any path in the edge union of the $\ASTG(N(I,x))$ for $x\in a$, hence any path in $\ASTG(N(I,a))$.

Lemma \ref{lem:reachability} follows immediately from the definitions of $N[I]$, $N(I,x)$ and $f^x_v$:

\begin{lem}
\label{lem:reachability}
Let $N=(G,f)$ be a Boolean network and $(I,V(G)\setminus I)$ be a decomposition of $N$. If $x,x'\in\B^I$ and $y,y'\in\B^{V(G)\setminus I}$, then
$$
(x,y) \xrightarrow{N} (x',y') \Longleftrightarrow
\begin{cases}
y=y' \text{ and } x \xrightarrow{N[I]} x' & or \\
x=x' \text{ and } y \xrightarrow{N(I,x)} y'.&
\end{cases}
$$
\end{lem}

\begin{cor}
\label{cor:reachability}
Let $N=(G,f)$ be a Boolean network and $(I,V(G)\setminus I)$ be a decomposition of $N$. If $x,x'\in\B^I$ and $y,y'\in\B^{V(G)\setminus I}$, then
$$
(x,y) \chemin{N} (x',y') \Longrightarrow x \chemin{N[I]} x'.
$$
\end{cor}

\begin{proof}
By Lemma \ref{lem:reachability}, and induction on the minimal length of a path from $(x,y)$ to $(x',y')$ in \ASTG(N).
\end{proof}

\subsection{Decompositions and attractors}
\label{sec:att}

\begin{thm}
\label{thm:attractors}
Let $N$ be a Boolean network and $(I,V(G)\setminus I)$ be a decomposition of $N$. Then
$$
\attractors(N)
=
\{ a\times b
\mid
a\in\attractors(N[I]),
b\in\attractors(N(I,a))
\}
$$
and $\attractors(N)$ is therefore in bijective correspondence with the dependent sum construction
$$
\sum_{a\in\attractors(N[I])}\attractors(N(I,a))
=
\{ (a,b)
\mid
a\in\attractors(N[I]),
b\in\attractors(N(I,a))
\}.
$$
\end{thm}

\subsection{Example}
\label{sec:ex}

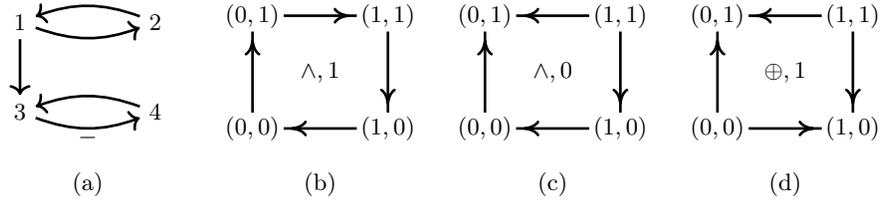
\begin{figure}[t]
\centering
\begin{subfigure}{0.24\textwidth}
\centering
\begin{tikzpicture}[scale=0.6,line width=1pt]
\node (3) at (0,0) {$3$};
\node (4) at (3,0) {$4$};
\node (2) at (3,2) {$2$};
\node (1) at (0,2) {$1$};
\draw[->] (1) -- (3);
\draw[->,bend right=20] (1) to (2);
\draw[->,bend right=20] (2) to (1);
\draw[->] (3) edge[bend right=20] node[midway,below=-3pt]{$-$} (4);
\draw[->,bend right=20] (4) to (3);
\end{tikzpicture}
\caption{}
\label{fig:a}
\end{subfigure}
\hfill
\begin{subfigure}{0.24\textwidth}
\centering
\begin{tikzpicture}[scale=0.6,line width=1pt]
\node at (1.5,1.25) {$\wedge,1$};
\node[inner sep=1pt] (00) at (0,0) {$(0,0)$};
\node[inner sep=1pt] (10) at (3,0) {$(1,0)$};
\node[inner sep=1pt] (11) at (3,2.5) {$(1,1)$};
\node[inner sep=1pt] (01) at (0,2.5) {$(0,1)$};
\draw[midarrow] (00) -- (01);
\draw[midarrow] (01) -- (11);
\draw[midarrow] (11) -- (10);
\draw[midarrow] (10) -- (00);
\end{tikzpicture}
\\
\vspace{1.5mm}
\caption{}
\label{fig:b}
\end{subfigure}
\hfill
\begin{subfigure}{0.24\textwidth}
\centering
\begin{tikzpicture}[scale=0.6,line width=1pt]
\node at (1.5,1.25) {$\wedge,0$};
\node[inner sep=1pt] (00) at (0,0) {$(0,0)$};
\node[inner sep=1pt] (10) at (3,0) {$(1,0)$};
\node[inner sep=1pt] (11) at (3,2.5) {$(1,1)$};
\node[inner sep=1pt] (01) at (0,2.5) {$(0,1)$};
\draw[midarrow] (00) -- (01);
\draw[midarrow] (11) -- (01);
\draw[midarrow] (11) -- (10);
\draw[midarrow] (10) -- (00);
\end{tikzpicture}
\\
\vspace{1.5mm}
\caption{}
\label{fig:c}
\end{subfigure}
\hfill
\begin{subfigure}{0.24\textwidth}
\centering
\begin{tikzpicture}[scale=0.6,line width=1pt]
\node at (1.5,1.25) {$\oplus,1$};
\node[inner sep=1pt] (00) at (0,0) {$(0,0)$};
\node[inner sep=1pt] (10) at (3,0) {$(1,0)$};
\node[inner sep=1pt] (11) at (3,2.5) {$(1,1)$};
\node[inner sep=1pt] (01) at (0,2.5) {$(0,1)$};
\draw[midarrow] (00) -- (01);
\draw[midarrow] (11) -- (01);
\draw[midarrow] (11) -- (10);
\draw[midarrow] (00) -- (10);
\end{tikzpicture}
\\
\vspace{1.5mm}
\caption{}
\label{fig:d}
\end{subfigure}
\caption{Example of Section \ref{sec:ex}: \subref{fig:a} interaction graph (positive interactions, except $f_4(x_3) = \neg x_3$); $\ASTG(N(I,(x_1,x_2)))$ on variables $3$ and $4$ for $f_3(x_1,x_4) = x_1\wedge x_4$ and $x_1=1$ \subref{fig:b}, $x_1=0$ \subref{fig:c}; \subref{fig:d} for $f_3(x_1,x_4)=x_1\oplus x_4$ and $x_1=1$.}
\label{fig:ex2}
\end{figure}

As an example, consider the Boolean network $N=(G,f)$ whose interaction graph $G$ is depicted in Figure \ref{fig:ex2}~\subref{fig:a}. Let $I=\{1,2\}$. If the Boolean functions are:
$$
\begin{array}{rlrl}
f_1(x_2) & = x_2 & \qquad f_2(x_1) & = x_1 \\
f_3(x_1,x_4) & = x_1\wedge x_4 & f_4(x_3) & = \neg x_3
\end{array}
$$
then $N[I]$ is a positive cycle between $1$ and $2$, and has two attractors, the fixed points $(0,0)$ and $(1,1)$. For $x_1=1$, $f_3(x_1,x_4)=x_4$, $N(I,(1,x_2))$ is a negative cycle between $3$ and $4$, and $\ASTG(N(I,(1,x_2)))$ is depicted in Figure \ref{fig:ex2}~\subref{fig:b}. For $x_1=0$, $f_3(x_1,x_4)=0$ and $\ASTG(N(I,(0,x_2)))$ is depicted in Figure \ref{fig:ex2}~\subref{fig:c}. Summing up, $N$ has two attractors: a cyclic attractor $\{(1,1)\}\times\B^{\{3,4\}}$ controlled by the fixed point $(1,1)$ of $N[I]$, and a fixed point $(0,0,0,1)$ controlled by the fixed point $(0,0)$ of $N[I]$.

If however $f_3(x_1,x_4)=x_1\oplus x_4$ is an exclusive or, then for $x_1=0$, $f_3(x_1,x_4)=x_4$. For $x_1=1$, $f_3(x_1,x_4)=\neg x_4$, hence $N(I,(1,x_2))$ is a positive cycle between $3$ and $4$ and $\ASTG(N(I,(1,x_2)))$ is depicted in Figure \ref{fig:ex2}~\subref{fig:d}. Then $N$ has three attractors: a cyclic attractor $\{(0,0)\}\times\B^{\{3,4\}}$ controlled by the fixed point $(0,0)$ of $N[I]$, and two fixed points $(1,1,0,1)$ and $(1,1,1,0)$ controlled by the fixed point $(1,1)$ of $N[I]$.

\subsection{Proof of Theorem \ref{thm:attractors}}
\label{sec:proof}

Let $\bigotimes_{a\in\attractors(N[I])}\attractors(N(I,a))$ denote the right-hand side $\{a\times b \mid a\in\attractors(N[I]), b\in\attractors(N(I,a))\}$.

\begin{enumerate}
\item Let us first prove that $\bigotimes_{a\in\attractors(N[I])}\attractors(N(I,a)) \subseteq \attractors(N)$. Let $a\in\attractors(N[I])$ and $b\in\attractors(N(I,a))$. To show that $a\times b$ is an attractor, we need to show that it is a maximal strongly connected subgraph.
\begin{enumerate}
\item We first prove strong connectedness: if $(x,y)$ and $(x',y')\in a\times b$, then $\ASTG(N)$ has a directed path from $(x,y)$ to $(x',y')$. Since $x$ and $x'$ both belong to the attractor $a\in\attractors(N[I])$, we have a path
$$
x\chemin{N[I]} x'
$$
which by Lemma \ref{lem:reachability} gives in particular a path
\begin{equation}
\label{eq:part2}
(x,y') \chemin{N} (x',y').
\end{equation}
On the other hand, since $y,y'\in b\in\attractors(N(I,a))$, we have a path
$$
y=y^0 \xrightarrow{N(I,a)} y^1 \xrightarrow{N(I,a)} \cdots \xrightarrow{N(I,a)} y^p=y'
$$
from $y$ to $y'$. By definition of $N(I,a)$, for each $i\in\{0,\ldots,p-1\}$, there exists $z^i\in a$ such that
$$
y^i \xrightarrow{N(I,z^i)} y^{i+1}.
$$
Since $z^i$ and $z^{i+1}$ both belong to the attractor $a\in\attractors(N[I])$, we also have a path
$$
z^i \chemin{N[I]} z^{i+1},
$$
and by combining these two paths, we obtain by Lemma \ref{lem:reachability} a path
$$
(z^i,y^i) \xrightarrow{N} (z^i,y^{i+1}) \chemin{N} (z^{i+1},y^{i+1})
$$
for each $i\in\{0,\ldots,p-1\}$. Since moreover $x\in a\in\attractors(N[I])$, we also have
$$
x \chemin{N[I]} z^0 \quad \text{and} \quad z^p \chemin{N[I]} x
$$
hence
\begin{equation}
\label{eq:part1}
(x,y)=(x,y^0) \chemin{N} (z^0,y^0) \chemin{N}\cdots\chemin{N}
(z^p,y^p) \chemin{N} (x,y^p)=(x,y').
\end{equation}
By combining paths (\ref{eq:part2}) and (\ref{eq:part1}), we get a directed path from $(x,y)$ to $(x',y')$ in $\ASTG(N)$, as expected.

\item Let us now prove maximality: if $(x,y)\in a\times b$ and
$$
(x,y) \xrightarrow{N} (x',y')
$$
for some $(x',y')\in\B^I\times\B^{V(G)\setminus I}$, then $(x',y')\in a\times b$. By Lemma \ref{lem:reachability}:
\begin{itemize}
\item either $y=y'$ and
$$
x \xrightarrow{N[I]} x'
$$
and then $x\in a\in\attractors(N[I])$ implies $x'\in a$;
\item or $x=x'$ and
$$
y \xrightarrow{N(I,x)} y' \quad \text{hence} \quad y \xrightarrow{N(I,a)} y'
$$
by definition of $N(I,a)$, and $y\in b\in\attractors(N(I,a))$ implies $y'\in b$.
\end{itemize}
In either case, $(x',y')\in a\times b$, as expected. This terminates the proof that $a\times b$ is an attractor of $N$.
\end{enumerate}

\item Let us now prove that $\attractors(N) \subseteq \bigotimes_{a\in\attractors(N[I])}\attractors(N(I,a))$. Assume $c\in\attractors(N)$, and let
\begin{align*}
a & = \{x\in\B^I \mid (x,y)\in c \text{ for some } y\in\B^{V(G)\setminus I}\} \\
b & = \{y\in\B^{V(G)\setminus I} \mid (x,y)\in c \text{ for some } x\in\B^I\}.
\end{align*}
\begin{enumerate}
\item Let us first show that $c=a\times b$. Since by construction of $a$ and $b$, we have $c\subseteq a\times b$, we just have to show that $a\times b\subseteq c$. To this end, let $x\in a$ and $y\in b$: there exist $y'\in\B^{V(G)\setminus I}$ and $x'\in\B^I$ such that both $(x,y')$ and $(x',y)\in c$. Therefore
$$
(x,y') \chemin{N} (x',y) \chemin{N} (x,y').
$$
By Corollary \ref{cor:reachability}, we have in particular a path
$$
x' \chemin{N[I]} x
$$
and by Lemma \ref{lem:reachability} a path \grosmot{$(x',y) \chemin{N} (x,y)$}. Since $(x',y)\in c\in\attractors(N)$, this implies $(x,y)\in c$. Therefore $a\times b\subseteq c$.

\item We now show that $a\in\attractors(N[I])$. If $x\in a$ and
$$
x \xrightarrow{N[I]} x'
$$
for some $x'\in\B^I$, then for any $y\in b$, $(x,y)\in a\times b$ and by Lemma \ref{lem:reachability}:
$$
(x,y) \xrightarrow{N} (x',y)
$$
so that $(x',y)\in a\times b\in\attractors(N)$, hence $x'\in a$. On the other hand, given $x,x'\in a$, again for any $y\in b$, both $(x,y)$ and $(x',y)\in a\times b\in\attractors(N)$, therefore
$$
(x,y) \xrightarrow{N} (x',y).
\vspace{2mm}
$$
By Lemma \ref{lem:reachability}, this means \grosmot{$x \xrightarrow{N[I]} x'$}, and $a$ is therefore maximally strongly connected, hence an attractor of $N[I]$.

\item Finally, we show that $b\in\attractors(N(I,a))$. Assume $y\in b$ and for some $x\in a$:
$$
y \xrightarrow{N(I,a)} y' \quad \text{hence} \quad y \xrightarrow{N(I,x)} y'
$$
Then $(x,y) \xrightarrow{N} (x,y')$ and $(x,y)\in a\times b\in\attractors(N)$, hence $(x,y')\in a\times b$ and $y'\in b$. On the other hand, given $y,y'\in b$, for any $x\in a$, both $(x,y)$ and $(x,y')\in a\times b\in\attractors(N)$, therefore
$$
(x,y) \xrightarrow{N} (x,y').
$$
By Lemma \ref{lem:reachability}, this means
$$
y \xrightarrow{N(I,x)} y' \quad \text{hence} \quad y \xrightarrow{N(I,a)} y'.
$$
This terminates the proof that $b$ is an attractor of $N(I,a)$.
\end{enumerate}
\end{enumerate}

\section{Generalized decompositions and strong modules}
\label{sec:gendecomposition}

Generalizing Section \ref{sec:decomposition}, let $I_1,\ldots,I_k$ be an ordered partition of $V(G)$ such that, for all $i<j$, $G$ has no directed path from a vertex of $I_j$ to a vertex of $I_i$. We call the tuple $(I_1,\ldots,I_k)$ a \emph{generalized decomposition} of $N$.

A generalized decomposition induces Boolean networks as follows. Given
$$
(x^1,\ldots,x^{k-1})\in\B^{I_1}\times\cdots\times\B^{I_{k-1}}
$$
we define networks $N_1,\ldots,N_k$ and $R_1,\ldots,R_k$ by:
$$
\begin{array}{rll}
R_1 & = N & \\
N_i & = R_i[I_i] & \quad \rdelim\}{2}{*}[\text{ for $1\leq i\leq k-1$}] \\
R_{i+1} & = R_i(I_i,x^i) & \\
N_k & = R_k. &
\end{array}
$$
The networks $N_1,\ldots,N_k$, with vertex sets $I_1,\ldots,I_k$, shall be called the \emph{networks induced by $N,I_1,\ldots,I_k,x^1,\ldots,x^{k-1}$}.

\begin{figure}[t]
\centering
\begin{tikzpicture}[scale=0.6,line width=1pt]
\draw (0,1) rectangle (3,8);
\draw[line width=1pt] (0,4) -- (3,4);
\draw[line width=1pt] (0,6) -- (3,6);
\node at (1.5,7) {$I_1$};
\node at (1.5,5) {$I_2$};
\node at (1.5,2.7) {$\vdots$};
\node[left] at (0,7) {$N_1=R_1[I_1]$};
\node[left] at (0,5) {$N_2=R_2[I_2]$};
\draw[decorate,decoration={brace,amplitude=10pt}] (-0.2,1) -- (-0.2,4);
\node[left] at (-0.7,2.5) {$R_3=R_2(I_2,a_2)$};
\draw[decorate,decoration={brace,amplitude=10pt}] (-4.7,1) -- (-4.7,6);
\node[left] at (-5.2,3.5) {$R_2=R_1(I_1,a_1)$};
\draw[decorate,decoration={brace,amplitude=10pt}] (-9.2,1) -- (-9.2,8);
\node[left] at (-9.7,4.5) {$R_1=N$};
\node[right] at (3,7) {$a_1$};
\node[right] at (3,5) {$a_2$};
\end{tikzpicture}
\caption{Networks induced by a generalized decomposition.}
\label{fig:generalized}
\end{figure}
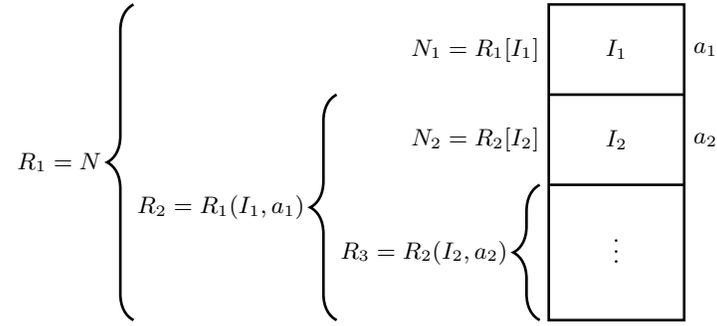

If $a_1\subseteq\B^{I_1},\ldots,a_{k-1}\subseteq\B^{I_{k-1}}$, we may define similarly $N_1,\ldots,N_k$ and $R_1,\ldots,R_k$ by letting $R_{i+1} = R_i(I_i,a_i)$ in the above definition. Again, the resulting networks $N_1,\ldots,N_k$, also with vertex sets $I_1,\ldots,I_k$, shall be called the \emph{networks induced by $N,I_1,\ldots,I_k,a_1,\ldots,a_{k-1}$} (Figure \ref{fig:generalized}).

\begin{prop}[Associativity]
\label{prop:ass}
Let $(I_1,I_2,I_3)$ be a generalized decomposition of a Boolean network $N$, $x^1\in\B^{I_1}$, $x^2\in\B^{I_2}$, $a_1\subseteq\B^{I_1}$ and $a_2\subseteq\B^{I_2}$. Then
\begin{align*}
N(I_1,x_1)[I_2] & = N[I_1\cup I_2](I_1,x_1) \\
N(I_1,x_1)(I_2,x_2) & = N(I_1\cup I_2,(x_1,x_2))
\end{align*}
and
\begin{align*}
N(I_1,a_1)[I_2] & = N[I_1\cup I_2](I_1,a_1) \\
N(I_1,a_1)(I_2,a_2) & = N(I_1\cup I_2,a_1\times a_2).
\end{align*}
\end{prop}

\begin{proof}
Consider the two equalities with $x_1,x_2$. Let $N=(G,f)$.
$$
\begin{array}{rlrl}
N(I_1,x_1) & = (G[I_2\cup I_3],(f^{x_1}_v)_{v\in I_2\cup I_3})
& \quad
N(I_1,x_1)[I_2] & = (G[I_2],(f^{x_1}_v)_{v\in I_2})
\\
N[I_1\cup I_2] & = (G[I_1\cup I_2],(f_v)_{v\in I_1\cup I_2})
& \quad
N[I_1\cup I_2](I_1,x_1) & = (G[I_2],(f^{x_1}_v)_{v\in I_2})
\end{array}
$$
See Figure \ref{fig:ass}. Moreover:
\begin{align*}
N(I_1,x_1)(I_2,x_2) & = (G[I_3],((f_v^{x_1})^{x_2})_{v\in I_3}) \\
& = (G[I_3],(f_v^{(x_1,x_2)})_{v\in I_3}) \\
& = N(I_1\cup I_2,(x_1,x_2))
\end{align*}
because $(f_v^{x_1})^{x_2}=f_v^{(x_1,x_2)}:y\in\B^{I_3}\mapsto f_v(x_1,x_2,y)$. Now, for any two networks $(G,f)$ and $(G,g)$ with the same interaction graph, we have:
$$
\ASTG(G,(f_v)_{v\in V(G)}) \cup \ASTG(G,(g_v)_{v\in V(G)}) = \ASTG(G,(h_v)_{v\in V(G)})
$$
where $\cup$ denotes here edge union, $h$ is defined by $h_v(x)\oplus x_v = (f_v(x)\oplus x_v) \vee (g_v(x)\oplus x_v)$ and $\oplus$ is exclusive or, and the two equalities with $a_1,a_2$ follow from the definition of $N(I,a)$.
\end{proof}

\begin{figure}[t]
\centering
\begin{tikzpicture}[scale=0.6,line width=1pt]
\draw (0,0) rectangle (3,6);
\draw[line width=1pt] (0,2) -- (3,2);
\draw[line width=1pt] (0,4) -- (3,4);
\node at (1.5,5) {$I_1$};
\node at (1.5,3) {$I_2$};
\node at (1.5,1) {$I_3$};
\draw[decorate,decoration={brace,amplitude=10pt}] (-0.2,0) -- (-0.2,4);
\node[left] at (-0.7,2) {$N(I_1,x_1)$};
\node[right] (A) at (3.5,2.2) {$N(I_1,x_1)[I_2]$};
\node at (6.5,3) {$=$};
\node[left] (B) at (9.5,3.8) {$N[I_1\cup I_2](I_1,x_1)$};
\node (AA) at (3,3) {};
\draw[->] (A.north) to[bend left=-25] (AA);
\node (BB) at (10,3) {};
\draw[->] (B.south) to[bend left=-25] (BB);
\draw (10,0) rectangle (13,6);
\draw[line width=1pt] (10,2) -- (13,2);
\draw[line width=1pt] (10,4) -- (13,4);
\node at (11.5,5) {$I_1$};
\node at (11.5,3) {$I_2$};
\node at (11.5,1) {$I_3$};
\draw[decorate,decoration={brace,mirror,amplitude=10pt}] (13.2,2) -- (13.2,6);
\node[right] at (13.7,4) {$N[I_1\cup I_2]$};
\end{tikzpicture}
\caption{An associativity relation from Proposition \ref{prop:ass}.}
\label{fig:ass}
\end{figure}
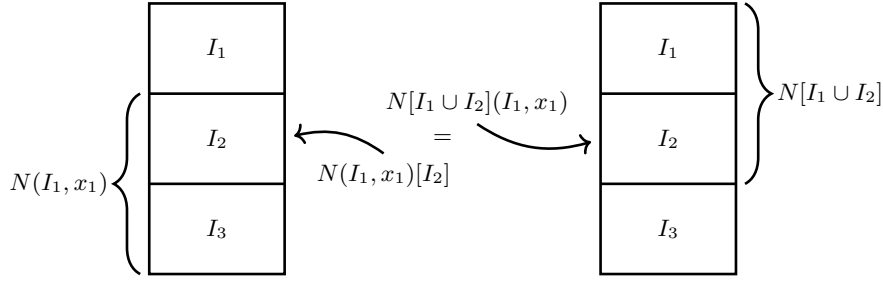

\begin{prop}[Commutativity]
\label{prop:com}
Let $(I_1,\ldots,I_k)$ be a generalized decomposition of a Boolean network $N=(G,f)$ and $i\in\{1,\ldots,k-1\}$ be such that $G$ has no edge from $I_i$ to $I_{i+1}$. Let $(x^1,\ldots,x^{k-1})\in\B^{I_1}\times\cdots\times\B^{I_{k-1}}$, $N_1,\ldots,N_k$ be the networks induced by
$$
N,I_1,\ldots,I_{i-1},I_i,I_{i+1},I_{i+2},\ldots,I_k,x^1,\ldots,x^{i-1},x^i,x^{i+1},x^{i+2},\ldots,x^{k-1}
$$
and $N'_1,\ldots,N'_k$ be the networks induced by
$$
N,I_1,\ldots,I_{i-1},I_{i+1},I_i,I_{i+2},\ldots,I_k,x^1,\ldots,x^{i-1},x^{i+1},x^i,x^{i+2},\ldots,x^{k-1}.
$$
Then $N_i=N'_{i+1}$, $N_{i+1}=N'_i$, and $N_j=N'_j$ for any $j\not\in\{i,i+1\}$. 

The similar statement for $a_1\subseteq\B^{I_1},\ldots,a_{k-1}\subseteq\B^{I_{k-1}}$ holds as well.
\end{prop}

\begin{proof}
This follows from Proposition \ref{lem:indep}.
\end{proof}

Propositions \ref{prop:ass} and \ref{prop:com} entail the following equalities by recurrence on $i$.

\begin{cor}
\label{prop:gendec}
With the above notations, we have, for $1\leq i\leq k$:
\begin{align*}
R_i & = N(I_1,a_1)\cdots(I_{i-1},a_{i-1}) \\
& = N(I_1\cup\cdots\cup I_{i-1},a_1\times\cdots\times a_{i-1}) \\
N_i & = N(I_1,a_1)\cdots(I_{i-1},a_{i-1})[I_i] \\
& = N[I_1\cup\cdots\cup I_{i}](I_1\cup\cdots\cup I_{i-1},a_1\times\cdots\times a_{i-1}).
\end{align*}
\end{cor}

\subsection{Strong modules}
\label{sec:strong}

Recall that the \emph{strong connectedness} equivalence relation $\sim$ on the vertex set $V(G)$ of a directed graph $G$ is defined as follows: for $u,v\in V(G)$, $u\sim v$ if and only if $G$ has directed paths from $u$ to $v$ and from $v$ to $u$. The equivalence classes of $\sim$ are the strongly connected components of $G$, and the \emph{condensation} of a directed graph $G$ is the directed acyclic graph $G/{\sim}$ with vertices the strongly connected components.

If $N=(G,f)$ is a Boolean network, we shall call \emph{strong modules} of $N$ the strongly connected components of $G$.

The condensation $G/{\sim}$ of $G$ induces an ordering on strongly connected components $H,H'$ of $G$: $H\leq H'$ when there is a directed path from (the vertex representing) $H$ to $H'$ in $G/{\sim}$. Extending this partial order to some linear order makes the vertex sets of strong modules of $N=(G,f)$ into a generalized decomposition of $N$. Conversely, if $(I_1,\ldots,I_k)$ is a generalized decomposition, the vertex set of any component of $G$ is included into some $I_i$.

\subsection{Generalized decompositions and attractors}

The following Theorem relates asynchronous attractors to generalized decompositions, in particular those arising from strong modules. It is proved by recurrence on $k$, using Corollary \ref{prop:gendec} and Theorem \ref{thm:attractors} (see Figure \ref{fig:attract}).

\begin{figure}[t]
\centering
\qquad\qquad\qquad
\begin{tikzpicture}[scale=0.6,line width=1pt]
\draw (0,0) rectangle (5,10);
\draw (0,4) -- (5,4);
\draw (0,6) -- (5,6);
\draw (0,8) -- (5,8);
\node at (1.5,9) {$N_1$};
\node at (1.5,7) {$N_2^{a_1}$};
\node at (1.5,5) {$N_3^{a_1,a_2}$};
\node at (1.5,2.2) {$\vdots$};
\node[right] at (2.6,8.8) {$a_1$};
\draw plot [smooth cycle] coordinates {
    (2.7,8.5)
    (3.7,8.35)
    (4.3,8.4)
    (4.3,8.6)
    (3.4,8.9)
    (2.9,9.7)
    (2.7,9.7)
};
\draw[->] (3.4,8.6) -- (3.6,7.6);
\draw[->] (3.5,8.5) to[bend left=60, looseness=2.5] (4.5,5);
\draw[->] (3.6,8.6) to[bend left=100, looseness=2] (4.5,3);
\node[right] at (2.6,6.8) {$a_2$};
\draw plot [smooth cycle] coordinates {
    (2.7,6.5)
    (3.7,6.35)
    (4.3,6.4)
    (4.3,6.6)
    (3.4,6.9)
    (2.9,7.7)
    (2.7,7.7)
};
\draw[->] (3.4,6.6) -- (3.6,5.6);
\draw[->] (3.6,6.6) to[bend left=80, looseness=1.7] (3.5,3);
\node[right] at (2.6,4.8) {$a_3$};
\draw plot [smooth cycle] coordinates {
    (2.7,4.5)
    (3.7,4.35)
    (4.3,4.4)
    (4.3,4.6)
    (3.4,4.9)
    (2.9,5.7)
    (2.7,5.7)
};
\draw[->] (3.4,4.6) -- (3.6,3.6);
\end{tikzpicture}
\caption{Attractors of $N$ as products of attractors of constrained networks $N_1$, $N_2^{a_1}$, $N_3^{a_1,a_2},\dots$ in Theorem \ref{thm:condensation}.}
\label{fig:attract}
\end{figure}
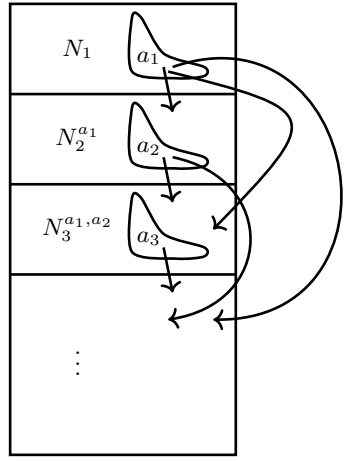

\begin{thm}
\label{thm:condensation}
Let $N$ be a Bool\-ean network, $(I_1,\ldots,I_k)$ be a generalized decomposition of $N$. Then $\attractors(N)$ is the set of Cartesian products $a_1\times\cdots\times a_k$ for
\begin{align*}
& a_1 \in \attractors(N_1) \\
& a_i \in \attractors(N_i^{a_1,\ldots,a_{i-1}}) \quad \text{if $2\leq i\leq k$},
\end{align*}
where
\begin{align*}
N_1 & = N[I_1] \\
N_i^{a_1,\ldots,a_{i-1}} & = N(I_1,a_1)\cdots(I_{i-1},a_{i-1})[I_i] \quad \text{if $1<i<k$} \\
N_k^{a_1,\ldots,a_{k-1}} & = N(I_1,a_1)\cdots(I_k,a_k),
\end{align*}
and $\attractors(N)$ is therefore in bijective correspondence with the dependent sum construction
\begin{align*}
& \sum_{a_1\in\attractors(N_1)}
\quad
\sum_{a_2\in\attractors\left(N_2^{a_1}\right)}
\quad
\cdots
\sum_{a_{k-1}\in\attractors\left(N_{k-1}^{a_1,\ldots,a_{k-2}}\right)}
\attractors(N_k^{a_1,\ldots,a_{k-1}}).
\end{align*}
\end{thm}

\subsection{Example}
\label{sec:ex3}

Let us consider again the example of Figure \ref{fig:ex2}, extended with a third layer: see Figure \ref{fig:ex3}~\subref{fig:aa}. Let $I_1=\{1,2\}$, $I_2=\{3,4\}$, $I_3=\{5,6\}$, and the Boolean functions be given by:
$$
\begin{array}{rlrl}
f_1(x_2) & = x_2 & \qquad f_2(x_1) & = x_1 \\
f_3(x_1,x_4) & = x_1\wedge x_4 & f_4(x_3) & = \neg x_3 \\
f_5(x_3,x_6) & = x_3\wedge x_6 & f_6(x_5) & = x_5.
\end{array}
$$
As in Figure \ref{fig:ex2}, $N_1=N[I_1]$ has two attractors $a_1=\{(0,0)\}$ and $a'_1=\{(1,1)\}$, $N_2^{a_1}=N(I_1,a_1)[I_2])$ has a single attractor $a_2=\{(0,1)\}$ and \grosmot{$N_2^{a'_1}=N(I_1,a'_1)[I_2])$} has a single attractor $a'_2=\B^{\{3,4\}}$. Now, $N_3^{a_1,a_2}=N(I_1,a_1)(I_2,a_2)[I_3])$ has a single attractor $a_3=\{(0,0)\}$. For \grosmot{$N_3^{a'_1,a'_2}=N(I_1,a'_1)(I_2,a'_2)[I_3])$}, note that
$$
\ASTG(N_3^{a'_1,a'_2})=\bigcup_{(x_3,x_4)\in a'_2}\ASTG(N_3^{a'_1,(x_3,x_4)}),
$$
which simply equals $\ASTG(N_3^{a'_1,(1,x_4)})$ and corresponds to a positive cycle between $5$ and $6$, therefore \grosmot{$N_3^{a'_1,a'_2}$} has two fixed points $a'_3=\{(0,0)\}$ and $a''_3=\{(1,1)\}$. Summing up, $N$ has three attractors:
\begin{align*}
a_1\times a_2\times a_3&=\{(0,0,0,1,0,0)\} \\
a'_1\times a'_2\times a'_3&=\{(0,0)\}\times\B^{\{3,4\}}\times\{(0,0)\} \\
a'_1\times a'_2\times a''_3&=\{(0,0)\}\times\B^{\{3,4\}}\times\{(1,1)\}.
\end{align*}
Condider now the variant with interaction graph depicted in Figure \ref{fig:ex3}~\subref{fig:bb} and $f_6(x_5) = x_5\vee x_4$. Then \grosmot{$N_3^{a_1,a_2}$} has a unique fixed point $a''_3=\{(0,1)\}$, and \grosmot{$N_3^{a'_1,a'_2}$}, depicted in Figure \ref{fig:ex3}~\subref{fig:cc}, has a single attractor $a'''_3=\{(0,0),(0,1),(1,1)\}$. Finally, $N$ has then two attractors $a_1\times a_2\times a''_3$ and $a'_1\times a'_2\times a'''_3$.

\begin{figure}[t]
\centering
\begin{subfigure}{0.3\textwidth}
\centering
\begin{tikzpicture}[scale=0.6,line width=1pt]
\node (3) at (0,0) {$3$};
\node (4) at (3,0) {$4$};
\node (2) at (3,2) {$2$};
\node (1) at (0,2) {$1$};
\node (5) at (0,-2) {$5$};
\node (6) at (3,-2) {$6$};
\draw[->] (1) -- (3) node[pos=0.9, above=2pt, right=1pt] {$\wedge$};
\draw[->] (3) -- (5) node[pos=0.9, above=2pt, right=1pt] {$\wedge$};
\draw[->,bend right=20] (1) to (2);
\draw[->,bend right=20] (2) to (1);
\draw[->] (3) edge[bend right=20] node[midway,below=-3pt]{$-$} (4);
\draw[->,bend right=20] (4) to (3);
\draw[->,bend right=20] (5) to (6);
\draw[->,bend right=20] (6) to (5);
\end{tikzpicture}
\caption{}
\label{fig:aa}
\end{subfigure}
\hfill
\begin{subfigure}{0.3\textwidth}
\centering
\begin{tikzpicture}[scale=0.6,line width=1pt]
\node (3) at (0,0) {$3$};
\node (4) at (3,0) {$4$};
\node (2) at (3,2) {$2$};
\node (1) at (0,2) {$1$};
\node (5) at (0,-2) {$5$};
\node (6) at (3,-2) {$6$};
\draw[->] (1) -- (3) node[pos=0.9, above=2pt, right=1pt] {$\wedge$};
\draw[->] (3) -- (5) node[pos=0.9, above=2pt, right=1pt] {$\wedge$};
\draw[->] (4) -- (6) node[pos=0.9, above=1pt, left=1pt] {$\vee$};
\draw[->,bend right=20] (1) to (2);
\draw[->,bend right=20] (2) to (1);
\draw[->] (3) edge[bend right=20] node[midway,below=-3pt]{$-$} (4);
\draw[->,bend right=20] (4) to (3);
\draw[->,bend right=20] (5) to (6);
\draw[->,bend right=20] (6) to (5);
\end{tikzpicture}
\caption{}
\label{fig:bb}
\end{subfigure}
\hfill
\begin{subfigure}{0.3\textwidth}
\centering
\begin{tikzpicture}[scale=0.6,line width=1pt]
\node[inner sep=1pt] (00) at (0,0) {$(0,0)$};
\node[inner sep=1pt] (10) at (2.8,0) {$(1,0)$};
\node[inner sep=1pt] (11) at (2.8,2.7) {$(1,1)$};
\node[inner sep=1pt] (01) at (0,2.7) {$(0,1)$};
\draw[midarrow] (00) -- (01);
\draw[midarrow] (01) -- (00);
\draw[midarrow] (01) -- (11);
\draw[midarrow] (11) -- (01);
\draw[midarrow] (10) -- (11);
\draw[midarrow] (10) -- (00);
\draw[dashed, thick, rounded corners]
([yshift=-1mm]00.south)
to ([xshift=-1mm,yshift=-1mm]00.south west)
to ([xshift=-1mm]00.west)
to ([xshift=-1mm]01.west)
to ([xshift=-1mm,yshift=1mm]01.north west)
to ([yshift=1mm]01.north)
to ([yshift=1mm]11.north)
to ([xshift=1mm,yshift=1mm]11.north east)
to ([xshift=1mm]11.east)
to ([xshift=1mm,yshift=-1mm]11.south east)
to ([yshift=-1mm]11.south)
to ([xshift=1mm,yshift=-1mm]01.south east)
to ([xshift=1mm]00.east)
to ([xshift=1mm,yshift=-1mm]00.south east)
-- cycle;
\end{tikzpicture}
\\
\vspace{4mm}
\caption{}
\label{fig:cc}
\end{subfigure}
\caption{Example of Section \ref{sec:ex3}: \subref{fig:a} interaction graph; \subref{fig:b}--\subref{fig:c} variant of interaction graph and \grosmot{$\ASTG(N_3^{a'_1,a'_2})$} on variables $5$ and $6$, with attractor $a'''_3$ indicated by a dashed loop.}
\label{fig:ex3}
\end{figure}
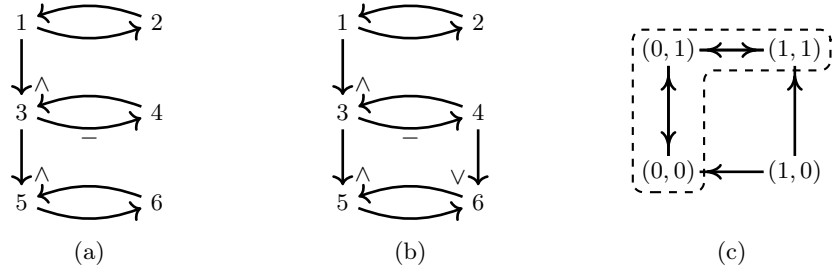

\section{Complexity}
\label{sec:complexity}

Theorem \ref{thm:condensation} tells us in particular that attractors of an $n$-dimensional network $N$ are Cartesian products $a_1\times\cdots\times a_k$. It therefore enables to avoid the exponential blow-up inherent in enumerating product-like structures by using factorized representations.

\subsection{Sparse Boolean networks with bounded strong modules}

A family of Boolean networks (of all dimensions) will be said to have \emph{bounded strong modules} when strong modules have dimension bounded by a constant (independent of dimension), and it will be said to be \emph{sparse} if the interaction graphs have in-degree bounded by a constant.

\begin{thm}
\label{th:complexity}
Sparse Boolean networks with bounded strong modules have (factorized representations of) attractors computable in polynomial time in the network's dimension.
\end{thm}

\begin{proof}
Let $N=(G,f)$ be an $n$-dimensional Boolean network, and $(I_1,\ldots,I_k)$ be a generalized decomposition of $N$, where the $I_i$ are vertex sets of strongly connected components of $G$. This can be computed in quadratic time (in $n$) using Tarjan’s algorithm. Then we compute:
\begin{itemize}
\item $N[I_1]$ in quadratic time;
\item $\attractors(N[I_1])$ in constant time, as the strong module $N[I_1]$ is assumed to have dimension bounded by some constant $c$;
\item for each of the at most $2^c$ attractors of $N[I_1]$, say $a_1\in\attractors(N[I_1])$, the network $N(I_1,a_1)$: for this, with compute at most $2^c$ networks $N(I_1,x)$ for $x\in a_1$ and take edge union; for each $x$, computing $N(I_1,x)$ takes quadratic time for the interaction graph and constant time for the Boolean functions in the sparse Boolean network $N$; summing up, it takes linear time.
\end{itemize}
Such a sequence has to be repeated $k\leq n$ times by Theorem \ref{thm:condensation}, whence a total time in $\mathcal{O}(n^3)$.
\end{proof}

\subsection{Polynomial size circuits and nested canalization}
\label{sec:canal}

Instead of considering sparse Boolean networks, one may require that the Boolean functions $f_v:\B^{G^-(v)}\to\B$, $v\in V(G)$, defining the network $N=(G,f)$ be representable by small circuits. For instance by polynomial circuits, \ie Boolean circuits of size bounded by $P(|G^-(v)|)$ for some fixed polynomial $P$ (independent of $n$). Call these networks Boolean networks \emph{with polynomial size circuits}.

There is an interesting class of Boolean networks with polynomial size circuits, namely \emph{nested canalizing} Boolean networks: those networks $N=(G,f)$ such that for each $v\in V(G)$, the function $f_v:\B^{G^-(v)}\to\B$ is nested canalizing. Recall that a Boolean function $f:\B^X\to\B$, with $|X|=d$, is said to be \emph{nested canalizing} if there exists a bijection $\sigma:\{1,\dots,d\}\to X$, along with canalizing input values $a_1, \dots, a_d \in \B$, canalized output values $b_1, \dots, b_d \in \B$, and a final output $b_{d+1}\in\B$, such that
$$
f(x) =
\begin{cases}
b_1 & \text{if } x_{\sigma(1)} = a_1 \\
b_2 & \text{if } x_{\sigma(1)} \ne a_1 \text{ and } x_{\sigma(2)} = a_2 \\
\;\vdots \\
b_d & \text{if } x_{\sigma(1)} \ne a_1, \dots, x_{\sigma(d-1)} \ne a_{d-1} \text{ and } x_{\sigma(d)} = a_d \\
b_{d+1} & \text{if } x_{\sigma(1)} \ne a_1, \dots, x_{\sigma(d)} \ne a_d.
\end{cases}
$$
Indeed, such a representation of $f_v$ can be viewed as a linear size Boolean circuit.

With the small circuits condition, computing the Boolean functions defining $N(I_1,a_1)$ in the proof of Theorem \ref{th:complexity} takes polynomial time too.

\begin{thm}
\label{th:complexityNC}
Boolean networks with polynomial size circuits and bounded strong modules (in particular nested canalizing Boolean networks with bounded strong modules) have (factorized representations of) attractors computable in polynomial time in the network's dimension.
\end{thm}

%

\section{Example of ERBB receptor-regulated G1/S transition}
\label{sec:bioex}

We consider the Boolean model of ERBB-driven signaling controlling the G1/S cell-cycle transition, as presented in \cite{SFLKBMMSCTPWBA09}, which is used to study and predict mechanisms of trastuzumab resistance. It models how signals from the ERBB receptor family regulate the transition from G1 phase to S phase of the cell cycle in cancer cells.
It involves $20$ variables: EGF, ERBB1, ERBB2, ERBB3, ERBB12, ERBB13, ERBB23, IGF1R, ER$\alpha$, cMYC, AKT1, MEK1, CDK2, CDK4, CDK6, CyclinD1, CyclinE1, p21, p27, pRB. EGF is an input of the network, and the regulatory functions of the other $19$ variables are recalled in Appendix \ref{app:rules}.

The interaction graph is represented in Figure 3 of \cite{SFLKBMMSCTPWBA09}. It has $12$ trivial strong modules (with a single vertex), and $2$ strong modules with $4$ vertices. The condensation is shown in Figure \ref{fig:condensation}, where $0$ represents the input EGF and the other genes are represented by a number from $1$ to $19$, as in Appendix \ref{app:rules}. The $2$ strong modules with $4$ vertices are represented by
\begin{align*}
A & = \{7,8,10,11\}=\{\text{IGF1R},\text{ER$\alpha$},\text{AKT1},\text{MEK1}\} \\
B & = \{12,13,17,18\}=\{\text{CDK2},\text{CDK4},\text{p21},\text{p27}\}.
\end{align*}
The small strong modules enable to compute all asynchronous attractors. Propagating attractors top-down in the consensation, we find that if the input EGF is $0$ (resp. $1$), the only attractor for the top $7$ variables fixes them to $0$ (resp. $1$). Then the controlled subnetwork $N_A^0$ (resp. $N_A^1$) on $A$ simplifies and is represented in Figure \ref{fig:condensation}. Propagating further down to the subnetworks $N_B^0$ and $N_B^1$ on $B$, we find $3$ global attractors, which turn out to be fixed points: with the $01$-notation in the ordering of genes defined above, they are $00000000000000000000$ (all genes inactive), $00000001111111111001$ (EGF is inactive but strong module $A$ stabilizes in an active state, which propagates down to pRB), and $11111111111111111001$ (EGF is active and forces pRB).


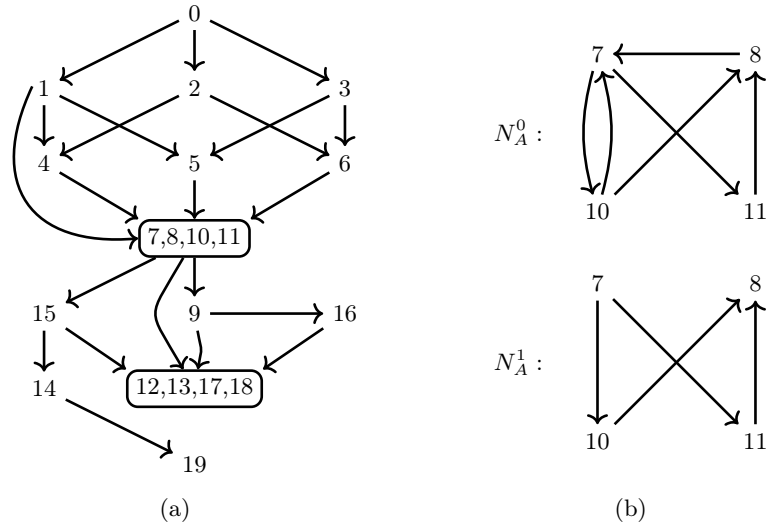
\begin{figure}[t]
\centering
\quad\quad
\begin{subfigure}{0.45\textwidth}
\centering
\begin{tikzpicture}[scale=1,line width=1pt]
\node (0) at (0,7) {0};
\node (1) at (-2,6) {1};
\node (2) at (0,6) {2};
\node (3) at (2,6) {3};
\node (4) at (-2,5) {4};
\node (5) at (0,5) {5};
\node (6) at (2,5) {6};
\node[draw, rounded corners] (A) at (0,4) {7,8,10,11};
\node (15) at (-2,3) {15};
\node (9) at (0,3) {9};
\node (16) at (2,3) {16};
\node[draw, rounded corners] (B) at (0,2) {12,13,17,18};
\node (14) at (-2,2) {14};
\node (19) at (0,1) {19};
\draw[->] (0) -- (1);
\draw[->] (0) -- (2);
\draw[->] (0) -- (3);
\draw[->] (1) -- (4);
\draw[->] (1) -- (5);
\draw[->] (2) -- (4);
\draw[->] (2) -- (6);
\draw[->] (3) -- (5);
\draw[->] (3) -- (6);
\draw[->] ([xshift=0.5mm,yshift=2.5mm]1.south west) to[bend right=60, looseness=1.5] (A.west);
\draw[->] (4) -- (A.north west);
\draw[->] (5) -- (A);
\draw[->] (6) -- (A.north east);
\draw[->] (A) -- (15);
\draw[->] (A) -- (9);
\draw[->] (A) to[bend right=30, looseness=1.7] (B);
\draw[->] (9) -- (16);
\draw[->] (9) to[bend left=10, looseness=1.7] (B);
\draw[->] (15) -- (B.north west);
\draw[->] (16) -- (B.north east);
\draw[->] (15) -- (14);
\draw[->] (14) -- (19);
\end{tikzpicture}
\caption{}
\label{fig:a}
\end{subfigure}
\hfill
\begin{subfigure}{0.45\textwidth}
\centering
\begin{tikzpicture}[scale=0.7,line width=1pt]
\node at (-1.5,1.5) {$N_A^0:$};
\node (10) at (0,0) {10};
\node (11) at (3,0) {11};
\node (8) at (3,3) {8};
\node (7) at (0,3) {7};
\draw[->] (8) -- (7);
\draw[->] (11) -- (8);
\draw[->] (10) -- (8);
\draw[->] (7) -- (11);
\draw[->] (7) to[bend right=15] (10);
\draw[->] (10) to[bend right=15] (7);
\end{tikzpicture}
\\[5mm]
\begin{tikzpicture}[scale=0.7,line width=1pt]
\node at (-1.5,1.5) {$N_A^1:$};
\node (10) at (0,0) {10};
\node (11) at (3,0) {11};
\node (8) at (3,3) {8};
\node (7) at (0,3) {7};
\draw[->] (11) -- (8);
\draw[->] (10) -- (8);
\draw[->] (7) -- (11);
\draw[->] (7) -- (10);
\end{tikzpicture}
\\[3mm]
\caption{}
\label{fig:b}
\end{subfigure}
\caption{Condensation and subnetworks $N_A^0,N_A^1$.}
\label{fig:condensation}
\end{figure}

\bibliographystyle{plain}


\newpage

\appendix
\section{Regulatory functions of the G1/S transition model}
\label{app:rules}

We recall here the regulatory functions of the G1/S transition model presented in \cite{SFLKBMMSCTPWBA09}, together with variable numbering. For each target variable, the input variables are represented by their number.

\begin{center}
\setlength{\tabcolsep}{2pt}
\begin{tabular}{rl@{\hspace{10pt}}l@{\hspace{20pt}}rl@{\hspace{10pt}}l}
\multicolumn{2}{c}{target} & Boolean function & \multicolumn{2}{c}{target} & Boolean function
\\
\hline
1 = &
ERBB1 & 0
&
2 = &
ERBB2 & 0
\\
3 = &
ERBB3 & 0
&
4 = &
ERBB12 & 1 \textwedge 2
\\
5 = &
ERBB13 & 1 \textwedge 3
&
6 = &
ERBB23 & 2 \textwedge 3
\\
7 = &
IGF1R & (8 \textvee 10) \textwedge \textneg 6
&
8 = &
ER$\alpha$ & 10 \textvee 11
\\
9 = &
cMYC & 8 \textvee 10 \textvee 11
&
10 = &
AKT1 & 1 \textvee 4 \textvee 5 \textvee 6 \textvee 7
\\
11 = &
MEK1 & 1 \textvee 4 \textvee 5 \textvee 6 \textvee 7
&
12 = &
CDK2 & 16 \textwedge \textneg 17 \textwedge \textneg 18
\\
13 = &
CDK4 & 15 \textwedge \textneg 17 \textwedge \textneg 18
&
14 = &
CDK6 & 15
\\
15 = &
CyclinD1 & 8 \textwedge 9 \textwedge (10 \textvee 11)
&
16 = &
CyclinE1 & 9
\\
17 = &
p21 & \multicolumn{4}{@{\hspace{0pt}}l}{%
8 \textwedge \textneg 10 \textwedge \textneg 9 \textwedge \textneg 13
}
\\
18 = &
p27 & \multicolumn{4}{@{\hspace{0pt}}l}{%
8 \textwedge \textneg 10 \textwedge \textneg 9 \textwedge \textneg 13 \textwedge \textneg 12
}
\\
19 = &
pRB & \multicolumn{4}{@{\hspace{0pt}}l}{%
13 \textwedge 14
}
\\[1mm]
\end{tabular}
\end{center}
\end{document}